%% file: equiv_rat_ambidex.tex
\documentclass{article}
\input{preamble.tex}
\usepackage[nottoc]{tocbibind}

\title{Ambidexterity in Rational Equivariant Stable Homotopy Theory}

\begin{document}

\maketitle 
\begin{abstract}
    We show that for a finite group $G$, the $G$-category $\RSp$ of rational genuine $G$-spectra is parametrized semiadditive, in the sense of Cnossen--Lenz--Linskens, for morphisms of genuine $G$-spaces whose fixed point maps have $\pi$-finite fibers.
    In particular, this means that the canonical norm map between a colimit and a limit of rational genuine $G$-spectra indexed by such $G$-spaces is an equivalence.
    This extends the Wirthmüller isomorphism from $G$-orbits to a broader class of $G$-spaces and combines the theory of ambidexterity of Hopkins and Lurie with parametrized category theory.
\end{abstract}

\tableofcontents 

\input{introduction.tex}
\input{recollections}

\input{main_result}

\bibliographystyle{alpha}
\bibliography{/Users/antonengelmann/Documents/Mathe/PhD/TeXDateien/Equivariant_rational_ambidextarity/Arxiv_version/equiv_rat_ambidex.bib}

\end{document}

%% file: preamble.tex
\usepackage[utf8]{inputenc}
\usepackage[english]{babel}
\usepackage[margin=1in]{geometry}

\usepackage{longtable}

\usepackage{amsmath, amsthm, amssymb, amsfonts}
\usepackage{mathtools}
\usepackage{ bbold }

\usepackage{contour}
\usepackage[normalem]{ulem}

\contourlength{0.5pt}

\newcommand{\myuline}[1]{%
  \uline{\phantom{#1}}%
  \llap{\contour{white}{#1}}%
}

\makeatletter
\newcommand*{\saved@myuline}{}
\let\saved@myuline\myuline

\newcommand*{\mathuline}{%
  \mathpalette{\math@myuline\saved@myuline}%
}
\newcommand*{\math@myuline}[3]{%
  \mbox{#1{$#2#3\m@th$}}%
}

\renewcommand*{\myuline}{%
  \relax  
  \ifmmode
    \expandafter\mathuline
  \else
    \expandafter\saved@myuline
  \fi
}

\usepackage{thmtools}
\usepackage[dvipsnames, svgnames]{xcolor}
\definecolor{vert}{RGB}{15,120,5}
\definecolor{gris}{RGB}{128,128,128}
\definecolor{bleu}{RGB}{0,50,150}
\usepackage[hypertexnames=false]{hyperref}
\hypersetup{
    colorlinks=true,
    linkcolor={bleu},
    citecolor= {vert},
    urlcolor={blue!80!black},
    pdfborder={0 0 0}
}
\usepackage[capitalize,noabbrev,nameinlink]{cleveref}
\newcommand{\defemph}[1]{\textcolor{purple}{\emph{#1}}}

\newcommand{\rsubsection}[1]{%
  \par\medskip
  \noindent\textbf{#1.}\quad
}

\usepackage{tikz-cd}
\usepackage{quiver}

\usepackage{hyperref}
\usepackage[capitalize,noabbrev]{cleveref}

\usepackage{enumitem}

\usepackage[colorinlistoftodos]{todonotes}
\usepackage{xpatch}
\usepackage[explicit]{titlesec}
\def\thissectiontitle{}
\def\thissectionnumber{}
\def\thissubsectiontitle{}
\def\thissubsectionnumber{}

\newtoggle{todoSection}
\newtoggle{todoSubsection}

\titleformat{\section}
  {\normalfont\Large\bfseries\sffamily}
  {\thesection}
  {0.5em}
  {\gdef\thissectiontitle{#1}\gdef\thissectionnumber{\thesection}#1}

\titleformat{\subsection}
  {\normalfont\large\bfseries\sffamily}
  {\thesubsection}
  {0.5em}
  {\gdef\thissubsectiontitle{#1}\gdef\thissubsectionnumber{\thesubsection}#1}

  \pretocmd{\section}{\global\toggletrue{todoSection}}{}{}
  \pretocmd{\subsection}{\global\toggletrue{todoSubsection}}{}{}

\AtBeginDocument{%
  \xpretocmd{\todo}{%
    \iftoggle{todoSubsection}{
     \addtocontents{tdo}{\protect\contentsline{subsection}%
        {\protect\numberline{\thissubsectionnumber}{\thissubsectiontitle}}{}{} }
      \global\togglefalse{todoSubsection}
        }{}
    }{}{}%
  }
\AtBeginDocument{%
  \xpretocmd{\todo}{%
    \iftoggle{todoSection}{
     \addtocontents{tdo}{\protect\contentsline{section}%
        {\protect\numberline{\thissectionnumber}{\thissectiontitle}}{}{} }
      \global\togglefalse{todoSection}
        }{}
    }{}{}%
  }
\definecolor{brightpink}{rgb}{1.0, 0.0, 0.5}

\theoremstyle{definition}
\newtheorem{defn}{Definition}[section]

\theoremstyle{definition}
\newtheorem{constr}[defn]{Construction}
\newtheorem{exmpl}[defn]{Example}

\theoremstyle{plain}
\newtheorem{lem}[defn]{Lemma}

\theoremstyle{plain}
\newtheorem{cor}[defn]{Corollary}

\theoremstyle{plain}
\newtheorem{prop}[defn]{Proposition}

\theoremstyle{plain}
\newtheorem{thm}[defn]{Theorem}

\theoremstyle{plain}
\newtheorem{introthm}{Theorem}

\theoremstyle{definition}
\newtheorem{introexmpl}[introthm]{Example}

\theoremstyle{plain}
\newtheorem{notation}[defn]{Notation}

\theoremstyle{remark}
\newtheorem{rmk}[defn]{Remark}

\crefname{lem}{Lemma}{Lemmas}
\crefname{defn}{Definition}{Definitions}
\crefname{cor}{Corollary}{Corollaries}
\crefname{prop}{Proposition}{Propositions}
\crefname{thm}{Theorem}{Theorems}
\crefname{rmk}{Remark}{Remarks}
\crefname{section}{Section}{Sections}
\crefname{equation}{Diagram}{Diagrams}
\crefname{introthm}{Theorem}{Theorems}

\newcommand{\Q}{{\mathbb{Q}}}

\newcommand{\1}{\mathbb{1}}

\newcommand{\colim}{\operatorname*{colim}}

\renewcommand{\lim}{\operatorname*{lim}}

\newcommand{\Fun}[2]{{\operatorname{Fun}({#1},{#2})}}

\newcommand{\PFun}[2]{{\operatorname{Fun^{\times}}({#1},{#2})}}
 
\newcommand{\GFun}[2]{{\operatorname{Fun}_G({#1},{#2})}} 
\newcommand{\uGFun}[2]{{\myuline{\operatorname{Fun}}_G({#1},{#2})}} 

\newcommand{\XS}{\myuline{\Sp}_G^{\myuline{\topos{X}}}} 

\newcommand{\SF}{{\operatorname{Span}({\Fin_G})}}

\newcommand{\Spc}{{\operatorname{Spc}}} 
\newcommand{\GSpc}{\Spc_{G}} 
\newcommand{\uGSpc}{\myuline{\Spc}_G} 

\newcommand{\Sp}{{\operatorname{Sp}}} 
\newcommand{\uSp}{\myuline{\Sp}_G} 
\newcommand{\RSp}{\myuline{\Sp}_{G,\Q}} 
\newcommand{\GSp}{{\Sp_G}} 

\newcommand{\Fin}{{\operatorname{Fin}}} 
\newcommand{\Orb}{{\operatorname{Orb}_G}} 

\newcommand{\Cati}{{\operatorname{Cat}_\infty}}

\newcommand{\GCat}{{\operatorname{Cat}_G}}

\newcommand{\PrL}{\mathcal{P}r^L}

\newcommand{\GPrL}{\PrL_G}

\newcommand{\op}[1]{#1^{\operatorname{op}}}

\newcommand{\topos}[1]{{\mathcal {#1}}}
\newcommand{\Cat}[1]{{\mathcal {#1}}}

\renewcommand{\P}[1]{\Cat{P}()}

\newcommand{\CAlg}[1]{{\operatorname{CAlg}(#1)}}

\newcommand{\id}[1]{\operatorname{id}_{#1}}

\newcommand{\Res}[2]{\operatorname{Res}^{#1}_{#2}}
\newcommand{\Coind}[2]{\operatorname{CoInd}^{#1}_{#2}}
\newcommand{\In}[2]{\operatorname{Ind}^{#1}_{#2}}
\newcommand{\Nm}{{\operatorname{Nm}}} 
\newcommand{\nm}{{\widetilde{\Nm}}} 

\usepackage{graphicx}

\newcommand{\cartsymb}{\arrow[dr, phantom,"\scalebox{1}{\color{black}$\lrcorner$}", near start, color=black]}

\author{Anton Engelmann\footnote{\href{https://anton-engelmann.com}{www.anton-engelmann.com}}}
\date{\today}

%% file: introduction.tex
\section{Introduction}
It is well known that for a finite group $G$ and a $G$-representation $V$ over $\Q$, there is a canonical isomorphism \[\Nm_G \colon V/G \xrightarrow{\simeq} V^G\]
between the $G$-coinvariants and the $G$-invariants given by $[v] \mapsto \sum_{g}^{} gv$.
If $E$ is a rational spectrum with the action of a finite group $G$, one can deduce from the above equivalence that the canonical norm map \[\Nm_G \colon E_{hG} \xrightarrow{\simeq} E^{hG} \in \Sp_\Q\]
is an equivalence.
Here, $E_{hG}$ denotes the homotopy orbits and $E^{hG}$ the homotopy fixed points.

There are several directions in which this result can be generalized.
One possible direction is to increase the chromatic height.
For this denote by $K(n)$ the Morava $K$-theory of height $n$, where $K(0)= \Q$, by $T(n)$ a telescope on a $v_n$-self map of some
type $n$ finite spectrum.
Recall that a space is called \defemph{$n$-finite} ($n \ge -2$) if it is $n$-truncated, has finitely many connected components, and all its homotopy groups are finite.
It is called \defemph{$\pi$-finite} if it is $n$-finite for some $n$, see \cite[Definition 4.4.1]{HopkinsLurie:Ambidexterity}.
Beginning with work of Hovey--Sadofsky and Greenlees--Sadofsky, several authors generalized this result, first by replacing rational spectra by $K(n)$-local spectra and then by replacing the classifying space $BG$ with arbitrary $\pi$-finite spaces.
These developments can be summarized as follows:

\begin{thm}[\cite{HoveySadofsky:TateCohomologyLowersBousfieldClasses,GreenleesSadofsky:TateSpectrumPeriodicComplexOriented,HopkinsLurie:Ambidexterity,CarmeliSchlankYanovski:AmidexterityChromatic}] \label{thm:introduction:HL}
    Let $X$ be a $\pi$-finite space and $A$ either $K(n)$ or $T(n)$.
    Denote by $\Sp_A$ the category of $A$-local spectra.
    Then for every $E \colon X \to \Sp_A$, there is a canonical equivalence \[\Nm_X \colon \colim_X E \xrightarrow{\simeq} \lim_X E \in \Sp_A.\]
\end{thm}
This may be rephrased as saying that $A$-local spectra are $\infty$-semiadditive or that $\pi$-finite spaces are $\Sp_A$-ambidextrous.

If $G$ is a finite group, then another possible generalization is to replace spaces by $G$-spaces and spectra by genuine $G$-spectra.
We write $\Orb$ for the $1$-category of transitive $G$-sets and $\GSp$ for the category of genuine $G$-spectra.
In stable equivariant homotopy theory, each subgroup $H \le G$ induces a restriction functor $\Res{G}{H} \colon \GSp \to \Sp_H$ from genuine $G$-spectra to genuine $H$-spectra.
The restriction admits a left and a right adjoint $\In{G}{H}$ and $\Coind{G}{H}$, respectively.
\begin{thm}[Wirthmüller isomorphism \cite{Wirthmueller:EquivariantHomologyDuality}]
    There is a canonical equivalence \[\In{G}{H} \xrightarrow{\simeq} \Coind{G}{H}.\]
\end{thm}
This is the equivariant analogue of the fact that finite coproducts and products coincide in additive categories.
There are several generalizations of this statement, e.g., to compact Lie groups \cite{Cnossen:TwistedAmbidexterityArticle} or global homotopy theory \cite{Lenz:GlobalHTTAlgebraicK,CnossenLenzLinskens:ParametrizedStabilityUPGlobalSpectra}.
To extend the result of Wirthmüller from $G$-orbits to $G$-spaces, one needs methods from parametrized homotopy theory.
Cnossen--Lenz--Linskens \cite{CnossenLenzLinskens:ParametrizedHigherSemiadditivityUniversality} extended Hopkins and Lurie's \cite{HopkinsLurie:Ambidexterity} notion of ambidexterity to categories internal to a topos.

\rsubsection{Main results}
Let $G$ be a finite group.
A $G$-category is a limit-preserving functor ${\Cat{C} \colon \op{\GSpc} \to \Cati}$.
The principal examples for this paper are the $G$-category $\uSp$ of genuine $G$-spectra, which informally sends an orbit $G/H$ to the category of genuine $H$-spectra $\Sp_H$, and its rationalization.
Within the theory of $G$-categories there are natural notions of $G$-(co)limits, i.e., (co)limits indexed by $G$-categories.

This naturally raises the question: For which $G$-spaces and $G$-categories do these $G$-limits and $G$-colimits coincide?

\noindent Wirthmüller answered this question for $\uSp$ and $G$-orbits.
The main goal of the present paper is to answer this question for the $G$-category $\RSp$ of genuine rational $G$-spectra.

\begin{introthm}[\cref{thm:main-theorem}] \label{introthm:semiadditivity}
    Let $G$ be a finite group and let $\topos Q$ be the wide subcategory of $\GSpc$ consisting of all the truncated morphisms whose fibers have $\pi$-finite fixed points.
    Then the $G$-category of rational genuine $G$-spectra, denoted by $\RSp$, is $\topos Q$-semiadditive.

    More precisely, let $(f\colon \topos{X \to Y}) \in \topos Q$.
    Then the restriction $f^\ast \colon \RSp(\topos Y) \to \RSp(\topos X)$ admits a left and a right adjoint $f_!$ and $f_\ast$, respectively, and the associated norm map \[\Nm_f \colon f_! \to f_\ast\] is an equivalence, and the same statement holds after every base change of $f$.
\end{introthm}
One can think of the class $\topos Q$ as the $G$-analogue of $\pi$-finite spaces.
Moreover, if $f \colon \topos X \to G/G$ is a morphism in $\topos Q$ and $E \colon \myuline{\topos X} \to \RSp$ is a diagram of $G$-categories, then \cref{introthm:semiadditivity} recovers the analogue of the  formula of \cref{thm:introduction:HL}:
\[\Nm_{\topos X} \colon \colim_{\topos X} E \xrightarrow{\simeq} \lim_{\topos X}E \in \Sp_{G,\Q}.\]
For example, if $\topos X = G/H$, then $E$ is a rational genuine $H$-spectrum and $\Nm_{G/H} \colon \In{G}{H} E \xrightarrow{\simeq} \Coind{G}{H}E$ recovers the Wirthmüller isomorphism. 

In the course of proving this theorem, we will rely on a reduction to fibers.
\begin{introthm}[\cref{lem:genuine-Q-fiberwise-norms}] \label{introthm:norms}
    Let $(f \colon \topos X \to \topos Y) \in \topos Q$ be a map of $G$-spaces such that the norm $\Nm_f$ exists.
    For every subgroup $H \le G$ denote points by $p_H \colon G/H \to \topos Y$.
    Write $f_{p_H} \colon F_{p_H} \to G/H$ for the pullback of $f$ along $p_H$.
    If the norm map \[\Nm_{f_{p_H}} \colon f_{p_H,!} \to f_{p_H,\ast}\] is an equivalence for each $f_{p_H}$, then the norm \[\Nm_f \colon f_! \to f_\ast\] is an equivalence.
\end{introthm}
This is an immediate corollary of the next result, which is of independent interest.
\begin{introthm}[\cref{lem:genuine-Q-fiberwise-ambidextrous}] \label{introthm:fibers}
    Let $G$ be a finite group and $(f\colon \topos{X \to Y}) \in \topos Q$.
    Then $f$ is $\RSp$-ambidextrous if and only if for every subgroup $H \le G$ and every point $p_H \colon G/H \to \topos Y$, the pullback of $f$ along $p_H$ is $\RSp$-ambidextrous.
\end{introthm}

In the case of spaces, the classifying spaces of finite groups give a large class of examples of $\pi$-finite spaces.
There is an analogue for $G$-spaces as well.

\begin{introexmpl} \label{lem:main-them:genuine-classifying-space}
    For compact Lie groups $G$ and $H$, Lashof \cite{Lashof:EquivariantBundles}, Lashof--May \cite{LashofMay:GeneralizedEquivariantBundles} and many others developed the theory of $G$-equivariant $H$-principal bundles.
    These admit a universal principal $(G,H)$-bundle with base $G$-space $B_G(H)$.

    Let us now assume that both $G$ and $H$ are finite.
    Then for $K \le G$, there is an equivalence \[B_G(H)^K \simeq \coprod_{[\Lambda]} B(H \cap N_{G \times H} \Lambda),\]
    where $[\Lambda]$ denotes the $H$-conjugacy classes of subgroups $\Lambda \le G \times H$ such that $\Lambda \cap H = e$ and $q(\Lambda) = K$ for the quotient map $q \colon G \times H \to G$.
    This is \cite[Theorem 10]{LashofMay:GeneralizedEquivariantBundles}.
    Therefore, the map $B_G(H) \to \ast$ is $\RSp$-ambidextrous by \cref{introthm:semiadditivity}.
\end{introexmpl}

\rsubsection{Proof strategy}
The usual strategy to prove statements about ambidexterity of maps is to use induction on the truncation level of these maps.
Let $(f \colon \topos X \to \topos Y) \in \topos Q$ be a map of $G$-spaces.
We wish to show that this map is $\RSp$-ambidextrous.
To do so, we show that for any base change $f' \colon \topos X' \to \topos Y'$ of $f$ the norm $\Nm_{f'}$ is an equivalence.
By \cref{introthm:fibers}, this reduces to a statement about fibers, similarly as in the proof of Hopkins--Lurie \cite[Proposition 4.3.5]{HopkinsLurie:Ambidexterity}.
That is, it will be enough to show that for all subgroups $H \le G$, the norm of the pullback of $f$ along any point $G/H \to \topos Y'$ is an equivalence.
In fact, it is enough to prove it for the pullback of $f$ along $G/G \to \topos Y'$.
The precise statements are spelled out in \cref{lem:ambidexterity:reduction-to-G/G,lem:main-thm:norm-to-point}.

A general method in stable equivariant homotopy theory is to apply geometric fixed points $\Phi^G$ and use that the family $\{\Phi^H\}_{H \le G}$ is jointly conservative.
This lets us exploit results from classical stable homotopy theory.
Here it will be crucial to work rationally, as we need to commute the geometric fixed points functor with limits.
This does not hold integrally.
Utilizing this we show that the norm for any map $\topos X \to \ast$ in $\topos Q$ is an equivalence.

\begin{rmk}
Rationality enters the proof of \cref{lem:main-thm:norm-to-point} at two points.
First, the morphism $\Phi^G f_\ast \to f^G_\ast \Phi^G$ is an equivalence, since rationally geometric fixed points preserve $G$-limits, see \cref{lem:geometric-fixed-points:parametrized-geometric-rationally-limits}.
This follows by the splitting of $\Sp_{G,\Q}$, where finiteness of $G$ is required.
Second, the nonequivariant norm map associated to $f^G \colon \topos X^G \to \ast$ is an equivalence since rational spectra are $\infty$-semiadditive by \cite[Theorem 5.2.1]{HopkinsLurie:Ambidexterity}
\end{rmk}

\rsubsection{Outline}
\cref{sec:preliminaries} contains the relevant background material.

In \cref{sec:parametrized-htt}, we review the notion of $G$-categories and recall the most common examples of $G$-spaces and $G$-spectra.
\cref{sec:geometric-fixed-points} reviews the constructions of the parametrized and non-parametrized geometric fixed points functors.
Moreover, it establishes that, rationally, both versions commute with the appropriate limits.
\cref{sec:recol-ambidexterity} contains the setup of \cite{HopkinsLurie:Ambidexterity,CnossenLenzLinskens:ParametrizedHigherSemiadditivityUniversality} for ambidexterity in the setting of $G$-categories.

In \cref{sec:class-of-morphisms}, we verify that our chosen class of morphisms of $G$-spaces satisfies the conditions of the preceding section.

\cref{sec:reduction-fibers} proves the reductions to fibers, namely \cref{introthm:norms,introthm:fibers}.

In \cref{sec:main-thm} we prove the main result, \cref{introthm:semiadditivity}.

\rsubsection{Notation}
This paper is written in the language of $\infty$-categories, as developed by, for example, \cite{Lurie:HigherTopos,Lurie:HigherAlgebra}.
We will write category for $\infty$-category and topos for $\infty$-topos.
A \defemph{space} is an object of the category $\Spc$ of spaces/homotopy types/$\infty$-groupoids/anima.
For an arbitrary finite group $G$, a \defemph{$G$-space} means an object in the topos $\GSpc$.
Here is a table of some of the notation that we use.
\begin{center}
\begin{longtable}{l|l|l}
notation & meaning & reference/definition \\
\hline
$G$ & some finite group & \\
$\Cati$ & (very large) category of categories & \\
$\Orb$ & orbit category of $G$  & \\
$\GSpc$ & category of $G$-spaces & $\Cat{P}(\Orb)$\\
$\topos X^H$ & $H$-fixed points of a $G$-space  & $\topos X(G/H)$\\
$\GSp$ & genuine $G$-spectra/spectral Mackey functors & $\PFun{\SF}{\Sp}$\\
$\uSp$ &$G$-category of genuine $G$-spectra & \cref{exmpl:parametrized:G-spectra} \\
$\RSp$ &$G$-category of rational genuine $G$-spectra & \cref{exmpl:parametrized:rational-G-spectra} \\
$\Phi$ & parametrized geometric fixed points functor &\cref{defn:geometric-fixed-points:parametrized-geometric-fixed} \\
$\topos{Q}$ &  & \cref{defn:genuine-Q-setup}
\end{longtable}
\end{center}

\rsubsection{Acknowledgments}
I thank Tom Bachmann, Julie Bannwart, Kaif Hilman, Dominik Kirstein, Klaus Mattis, Luca Passolunghi and Timo Weiß for very helpful discussions.
Furthermore, I thank Tom Bachmann and Klaus Mattis for reading a draft of this paper.

The author acknowledges funding by the Deutsche Forschungsgemeinschaft (DFG,
German Research Foundation) through the Collaborative Research Centre TRR 326 ‘Geometry and Arithmetic of Uniformized Structures’, project number 444845124.

%% file: recollections.tex
\section{Preliminaries} \label{sec:preliminaries}
In this section, we recall the background material needed for the proof of the main theorem.
In \cref{sec:parametrized-htt} we review $G$-categories and some examples such as $G$-spaces and genuine $G$-spectra.
\cref{sec:geometric-fixed-points} recalls (parametrized) geometric fixed points.
In \cref{sec:recol-ambidexterity} we review ambidexterity and semiadditivity.

\subsection{$G$-Categories} \label{sec:parametrized-htt}
The variant of ambidexterity used in this paper can most conveniently be phrased in terms of $G$-categories, a special case of parametrized category theory.
These notions were first developed for presheaf topoi by \cite{BarwickDottoGlasmanNardinShah:GeneralIntroduction,BarwickDottoGlasmanNardinShah:ExposeI,Nardin:ExposeIV,Shah:ParametrizedHigherCategoryTheory} 
and were later generalized to category theory internal to a topos by \cite{Martini:YonedaLemmaInternalCategories,Martini:CocartesianFibrationsStraighteningInternal,MartiniWolf:ColimitsCocompletionsInternalHCT,MartiniWolf:PresentabilityTopoiInternalHigher}.

\begin{defn} \label{defn:parametrized:T-cats}
    \begin{enumerate}
        \item Let $T$ be a small category.
        By a \defemph{$T$-category} we mean a functor $\Cat{C} \colon \op{T} \to \Cati$.
        We write $\textup{Cat}_T$ for the (very large) category of $T$-categories.
        \item Let $\topos B$ be a topos.
        Then we define a \defemph{$\topos B$-category} to be a limit-preserving functor $\Cat{C} \colon \op{\topos B} \to \Cati$.
        We write $\textup{Cat}(\topos{B})$ for the (very large) category of $\topos B$-categories.
    \end{enumerate}
\end{defn}

\begin{rmk} \label{rmk:parametrized:G-cats}
    For a presheaf topos $\Cat{P}(T)$, the Yoneda embedding induces an equivalence ${\textup{Cat}(\Cat{P}(T)) \xrightarrow{\simeq} \textup{Cat}_T}$.
    In particular, if $G$ is a finite group, we will write $\GCat$ for $\textup{Cat}_\Orb \simeq \textup{Cat}(\GSpc)$ and refer to its objects as \defemph{$G$-categories}.
\end{rmk}

\begin{defn}[$\topos{Q}$-colimits] \label{defn:parametrized:Q-colimits}
    Let $\topos{Q}$ be a class of morphisms in $\GSpc$ closed under base change.
    A $G$-category $\Cat{C} \colon \op{\GSpc} \to \Cati$ is called \defemph{$\topos{Q}$-cocomplete} if
    \begin{enumerate}
        \item For every $q \colon \topos X \to \topos Y$ in $\topos{Q}$, the functor $q^\ast\colon \Cat{C}(\topos Y) \to \Cat{C}(\topos X)$ admits a left adjoint $q_! \colon \Cat{C}(\topos X) \to \Cat{C}(\topos Y)$.
        \item For every pullback square \begin{center}
            \begin{tikzcd}
                \topos X' \ar[r, "g"] \ar[d, "q'"] \cartsymb & \topos X \ar[d, "q"]\\
                \topos Y' \ar[r, "f"] & \topos Y
            \end{tikzcd}
        \end{center}
        in $\GSpc$ with $q \in \topos{Q}$, the Beck-Chevalley transformation $\mathrm{BC}_! \colon q'_! g^\ast \to f^\ast q_! \colon \Cat{C}(\topos X) \to \Cat{C}(\topos Y')$ is an equivalence.
    \end{enumerate}
\end{defn}
\noindent Dually, one may define what it means for $\Cat{C}$ to be \defemph{$\topos{Q}$-complete}.

\begin{defn} \label{defn:parametrized:G-presentable}
    A $G$-category $\Cat{C} \colon \op{\GSpc} \to \Cati$ is called \defemph{$G$-presentable} if it factors through $\PrL$ and is $\GSpc$-cocomplete.
\end{defn}
This is the most convenient definition for our purposes.
There are many equivalent versions, see \cite[Theorem 2.4.2.5]{MartiniWolf:PresentabilityTopoiInternalHigher}.

\begin{defn} \label{defn:parametrized:preserve-G-colimits}
    A functor $F \colon \Cat{C} \to \Cat{D}$ of $G$-presentable $G$-categories \defemph{preserves $G$-colimits} if
    \begin{enumerate}
        \item For any $\topos X \in \GSpc$, the map $F(\topos X) \colon \Cat{C}(\topos X) \to \Cat{D}(\topos X)$ preserves colimits.
        \item For any map $f \colon \topos X \to \topos Y$ in $\GSpc$ the Beck-Chevalley transformation $\mathrm{BC}_! \colon f_! F(\topos  X) \to F(\topos Y)f_!$ is an equivalence.
    \end{enumerate}
\end{defn}
If $\Cat{C}$ is $G$-presentable, then the restriction $f^\ast \colon \Cat{C}(\topos Y) \to \Cat{C}(\topos X)$ admits a right adjoint $f_\ast$.
In this situation the definition of $F$ \defemph{preserving $G$-limits} is dual to \cref{defn:parametrized:preserve-G-colimits}.

\begin{notation}
    We write $\GPrL$ for the (non-full) subcategory of $\GCat$ spanned by the $G$-presentable $G$-categories and the $G$-colimit-preserving functors.
\end{notation}

We now recall several examples of ($G$-presentable) $G$-categories.

\begin{exmpl} \label{exmpl:parametrized:single-G-space}
    Let $\topos X \in \GSpc$.
    Then this defines a $G$-category $\myuline{\topos X}$ via the Yoneda embedding \[\myuline{\topos X} \coloneqq \hom_{\GSpc}(-, \topos X) \colon \op{\GSpc} \to \Spc \to \Cati.\]
\end{exmpl}

\begin{exmpl} \label{exmpl:parametrized:G-spaces}
    Recall that the target functor $\GSpc^{\Delta^1} \to \GSpc$ is a cartesian fibration and hence, by \cite[Theorem 6.1.3.9]{Lurie:HigherTopos}, is classified by a limit preserving functor $\op{\GSpc} \to \PrL$
    sending an object $\topos X$ to $(\GSpc)_{/\topos X}$ and a morphism $(f \colon \topos{X \to Y})$ to the pullback functor $(f^\ast \colon (\GSpc)_{/\topos Y} \to (\GSpc)_{/\topos X})$.
    We denote this $G$-presentable $G$-category by $\uGSpc$. 
    See \cite[Definitions 2.11 and 4.2]{Cnossen:TwistedAmbidexterityArticle} and \cite{Shah:ParametrizedHigherCategoryTheory} for more details.
\end{exmpl}

Recall from \cite[Corollary 2.17]{Cnossen:TwistedAmbidexterityArticle} and the preceding discussion, that there is a fully faithful functor \[- \otimes_{\GSpc} \uGSpc \colon \textup{CAlg}_{\GSpc}(\PrL) \to \CAlg{\GPrL}\]
which sends a commutative $\GSpc$-algebra $\Cat{D}$ in $\PrL$ to a $G$-presentable symmetric monoidal $G$-category $\Cat{D} \otimes_{\GSpc} \uGSpc$.
On an object $\topos X \in \GSpc$, this is given by the relative tensor product $\Cat{D}\otimes_{\GSpc} (\uGSpc)_{/\topos{X}}$ in $\PrL$.

\begin{exmpl} \label{exmpl:parametrized:G-spectra}
    Since $\GSp$ is a presentably symmetric monoidal category equipped with a symmetric monoidal left adjoint from $\GSpc$, it is a commutative $\GSpc$-algebra in $\PrL$.
    Therefore, we define the $G$-presentably symmetric monoidal \defemph{$G$-category of genuine $G$-spectra} as $\uSp \coloneqq \GSp \otimes_{\GSpc} \uGSpc$.
    The value on an orbit $G/H$ is given by $\uSp(G/H) = \Sp_H$.
    Note that, as in the non-parametrized case, there is a suspension functor $\Sigma^\infty_+ \colon \uGSpc \to \uSp$.
    For further details, see also \cite[Definition 4.2]{Cnossen:TwistedAmbidexterityArticle}.
\end{exmpl}

\begin{exmpl} \label{exmpl:parametrized:rational-G-spectra}
    Let us denote by $\Sp_{G,\Q}$ the rationalization of genuine $G$-spectra, 
    i.e., the Bousfield localization of $\GSp$ at $\1_G \otimes \Q$.
    The category $\Sp_{G,\Q}$ is still an object in $\textup{CAlg}_{\GSpc}(\PrL)$, hence we define the $G$-presentably symmetric monoidal $G$-category of rational $G$-spectra as \[\RSp \coloneqq \Sp_{G,\Q} \otimes_{\GSpc} \uGSpc.\]
\end{exmpl}

\begin{exmpl} \label{exmpl:parametrized:internal-hom}
    Martini \cite[Proposition 3.2.11]{Martini:YonedaLemmaInternalCategories} showed that $\GCat$ is cartesian closed.
    We consequently denote its internal hom by $\uGFun{\Cat{C}}{\Cat{D}}$ and its underlying category by $\GFun{\Cat{C}}{\Cat{D}} \coloneqq \Gamma(\uGFun{\Cat{C}}{\Cat{D}})$, that is, by evaluating it at the terminal orbit $G/G$.
    
    For $\topos X \in \GSpc$ and $\Cat{C} \in \GCat$, there is an equivalence $\GFun{\myuline{\topos X}}{\Cat{C}} \simeq \Cat{C}(\topos X)$, see \cite[Corollary 2.2.8]{CnossenLenzLinskens:ParametrizedStabilityUPGlobalSpectra}.
    We write $\XS \coloneqq \uGFun{\myuline{\topos X}}{\uSp}$ for the equivariant local systems on $\myuline{\topos X}$.

\end{exmpl}

\subsection{Geometric fixed points} \label{sec:geometric-fixed-points}
In equivariant stable homotopy theory there are many types of fixed points each of which has its own advantages and disadvantages.
Geometric fixed points are useful for detecting equivalences, since the collection of geometric fixed points functor for all subgroups is jointly conservative.
In this section we revisit their classical definition on genuine $G$-spectra and their parametrized definition.
Moreover, we prove that the parametrized geometric fixed points on the $G$-category of rational $G$-spectra preserve $G$-limits.

\begin{rmk} \label{defn:geometric-fixed-points:geometric-fixed-points}
    The geometric fixed points functor can be constructed via the symmetric monoidal composition \[\Phi^G \colon \GSp \to \GSp/\Sp_G^\Cat{P} \xrightarrow{\simeq} \Sp\] where $\Sp_G^\Cat{P}$ is the localizing subcategory of $\GSp$ generated by the orbits $\Sigma^\infty_+ G/H$ of proper subgroups $H<G$.
    Here, the first map is a localization and the second is an equivalence.
    The other properties we will need are that $\Phi^G$ is a left adjoint functor and there is a symmetric monoidal equivalence $\Phi^G \circ \Sigma^\infty \simeq \Sigma^\infty \circ (-)^G$.
    If $H \le G$ is a subgroup, then $\Phi^H$ is defined as the composite $\GSp \xrightarrow{\Res{G}{H}} \Sp_H \xrightarrow{\Phi^H} \Sp$.
\end{rmk}

\begin{lem} \label{lem:geometric-fixed-points:rational-to-rational}
    Let us denote the rationalization functor by $L_\Q$.
    Then $\Phi^G L_\Q \simeq L_\Q \Phi^G$.
\end{lem}
\begin{proof}
    Both $\Sp$ and $\GSp$ are presentable additive categories, by \cite[Corollary B.23]{HillHopkinsRavenel:KervaireInvariantNonExistence} and \cite[Proposition 2.8]{GepnerGrothNikolaus:Universality}.
    The geometric fixed points functor is exact and a left adjoint.
    We conclude the proof with \cite[Corollary 2.7]{Mattis:UnstableArithmeticFractureSquare}.
\end{proof}

\begin{constr} \label{defn:geometric-fixed-points:parametrized-geometric-fixed}
    Hilman--Kirstein--Kremer construct a parametrized version of geometric fixed points.
    Let us recall their construction in the special case where $\Cat{P}$ is the family of proper subgroups of $G$, so that $\Cat{P}^c = \{G\}$.

    If $s \colon \ast \simeq \textup{Orb}_{\Cat{P}^c} \to \op{\mathrm{Orb}_G}$ is the inclusion of the orbit $G/G$, then there is an induced adjunction 
    \[\tilde{s}^\ast \colon \textup{Pr}^{L,\textup{st}}_{G} \rightleftarrows \textup{Pr}^{L,\textup{st}}_{G,\Cat{P}^c} \colon s_\ast\]
    which is a smashing localization, where the left adjoint $\tilde{s}^\ast$ is symmetric monoidal and the right adjoint $s_\ast$ is fully faithful.
    See \cite[Theorem 2.2.26]{HilmanKirsteinKremer:ParametrisedPoincareDualityEquivariant} for a more detailed statement.
    
    Then the parametrized geometric fixed point functor is defined as the unit map \[\Phi \coloneqq \Phi^G \colon \uSp \to s_\ast\tilde{s}^\ast \uSp \simeq \uSp^{\Phi \tilde{\Cat{P}}}.\]
    This is constructed in \cite[Construction 2.2.31, Example 4.2.1]{HilmanKirsteinKremer:ParametrisedPoincareDualityEquivariant}.
    Note that $\Phi$ preserves $G$-colimits by construction, and that it is possible to construct $\Phi^H$ for the $G$-category of $G$-spectra similarly to the non-parametrized version.
    Moreover, the functor of $G$-categories $\Phi^H$ recovers the classical functor $\Phi^H$ after taking global sections.
\end{constr}

\begin{lem} \label{lem:geometric-fixed-points:geometric-rationally-limits}
    Let $H \le G$ be a subgroup.
    The functor $\Phi^H \colon \Sp_{G,\Q} \to \Sp_\Q$ preserves limits. 
\end{lem}
\begin{proof}
    Since restriction is right adjoint to induction, it preserves limits.
    Therefore, it is enough to show that $\Phi^G$ preserves them.
    Rationally, the geometric fixed points witness the splitting \[\Sp_{G,\Q} \simeq \prod_{H \le G} \Fun{BW_G H}{\Sp_\Q}\] \cite[Theorem 3.10]{Wimmer:ModelGenuineEquivariantCommutativeRingSpectra} (for $R = \Q$ and $\Cat{F}$ the full family of subgroups of $G$).
    In particular, $\Phi^G$ identifies with the projection $\prod_{H \le G} \Fun{BW_G H}{\Sp_\Q}\to \Fun{BW_G G}{\Sp_\Q} \simeq \Sp_\Q$.
    Hence, as a projection it preserves limits \cite[\href{https://kerodon.net/tag/069N}{Tag 069N}]{kerodon}.
\end{proof}

\begin{lem} \label{lem:geometric-fixed-points:parametrized-geometric-rationally-limits}
    The functor $\Phi \colon \RSp \to \RSp^{\Phi \tilde{\Cat{P}}}$ preserves $G$-limits.
\end{lem}
\begin{proof}
    Note that source and target are $G$-stable\footnote{A $G$-presentable category is called \defemph{$G$-stable} if it is a module over the idempotent algebra $\uSp \in \CAlg{\GPrL}$, see \cite[Definition A.0.4]{HilmanKirsteinKremer:ParametrisedPoincareDualityEquivariant}.} $G$-presentable $G$-categories, hence $\Phi$ preserves finite $G$-limits as it preserves all $G$-colimits by construction.
    By the dual version of \cite[Proposition 4.7.1]{MartiniWolf:ColimitsCocompletionsInternalHCT}, it is enough to show that $\Phi$ preserves finite $G$-limits and fiberwise limits.
    The latter follows from \cref{lem:geometric-fixed-points:geometric-rationally-limits}.
\end{proof}

\subsection{Ambidexterity and semiadditivity} \label{sec:recol-ambidexterity}
The terminology of ambidexterity for categories was first introduced by Hopkins and Lurie in \cite[Construction 4.1.8]{HopkinsLurie:Ambidexterity}.
It was later refined to $\topos{B}$-categories by \cite{Nardin:ExposeIV,CnossenLenzLinskens:ParametrizedHigherSemiadditivityUniversality}.
In this section we review the definition of ambidexterity in the language of the latter.

\begin{defn}[Inductible subcategory, {\cite[Definition 3.1]{CnossenLenzLinskens:ParametrizedHigherSemiadditivityUniversality}}] \label{defn:recollection-ambidexterity:inductible}
    Let $\Cat{A}$ be a category and let $\topos{Q}\subset\Cat{A}$ be a wide subcategory closed under base change.
    The subcategory $\topos{Q}$ is called \defemph{inductible} if
    \begin{enumerate}
        \item $\topos{Q}$ is closed under diagonals: if $q \colon A \to B$ in $\topos{Q}$, then $\Delta_q \colon A \to A \times_B A$ in $\topos{Q}$.
        \item $\topos{Q}$ is truncated: every $q \colon A \to B$ in $\topos{Q}$ is truncated, i.e., $n$-truncated for some integer $n$.
    \end{enumerate}
\end{defn}

\begin{defn}[Locally inductible, {\cite[Definition 3.9]{CnossenLenzLinskens:ParametrizedHigherSemiadditivityUniversality}}] \label{defn:recollection-ambidexterity:locally-inductible}
    Let $\topos{B}$ be a topos and let $\topos{Q} \subset \topos{B} $ be a wide local (in the sense of \cite[Proposition 6.2.3.14]{Lurie:HigherTopos}) subcategory.
    The subcategory $\topos{Q}$ is \defemph{locally inductible} if
    \begin{enumerate}
        \item $\topos{Q}$ is closed under diagonals.
        \item Every morphism $q \colon A \to B$ in $\topos{Q}$ is locally truncated: there exists a covering $(B_i \to B)_{i \in I}$ such that each base change $q_i \colon B_i \times_B A \to B_i$ is truncated.
    \end{enumerate}
\end{defn}

\begin{exmpl}[$\pi$-finite spaces] \label{exmpl:recollection-ambidexterity:pi-finite-spaces}
    We denote the class of $\pi$-finite spaces by $\Spc_{\pi} \subset \Spc$, see \cite[Definition 4.4.1]{HopkinsLurie:Ambidexterity}.
    A map $f \colon X \to Y$ of spaces is \defemph{locally in $\Spc_\pi$} if for each point $y \colon \ast \to Y$ the fiber $X_y$ over $y$ is $\pi$-finite.
    Note that this coincides with the definition of \cite[Construction 3.26]{CnossenLenzLinskens:ParametrizedHigherSemiadditivityUniversality}.

    Indeed, let $f \colon X \to Y$ be a map of spaces locally in $\Spc_\pi$ and $B$ a $\pi$-finite space with a map $g \colon B \to Y$.
    Then for a point $b \colon \ast \to B$ the fiber $(X \times_Y B)_b$ agrees with the fiber $X_{g(b)}$ by pasting the pullbacks (e.g., \cite[\href{https://kerodon.net/tag/03FZ}{Tag 03FZ}]{kerodon})
    \begin{center}
        \begin{tikzcd}
            (X \times_Y B)_b \ar[d] \ar[r] \cartsymb & \ast \ar[d, "b"] \\
            X \times_Y B \ar[d] \ar[d] \ar[r] \cartsymb & B \ar[d,"g"] \\
            X \ar[r, "f"] & Y\rlap{.}
        \end{tikzcd}
    \end{center} 
    In particular, the space $(X \times_Y B)_b$ is $\pi$-finite.
    It follows that $X \times_Y B$ is $\pi$-finite since $\Spc_\pi$ is closed under extensions; see, for example, \cite[Proposition 2.2.3]{Anel:piFiniteSpaces}.

    The morphisms locally in $\Spc_\pi$ determine a wide subcategory $(\Spc_\pi)_{\mathrm{loc}}$ of $\Spc$.
    Moreover, $(\Spc_\pi)_{\mathrm{loc}}$ is locally inductible by combining \cite[Example 3.32 and Lemma 3.28]{CnossenLenzLinskens:ParametrizedHigherSemiadditivityUniversality}.
\end{exmpl}

\begin{defn}[$\Cat{C}$-ambidexterity] \label{defn:recollection-ambidexterity:n-ambidexterity}
    Let $\Cat{A}$ be a category and $\topos{Q}$ an inductible subcategory of $\Cat{A}$.
    Assume that $\Cat{C} \colon \op{\Cat{A}} \to \Cati$ is $\topos{Q}$-cocomplete. 
    One inductively defines when an $n$-truncated morphism $(f \colon A \to B) \in \topos Q$ is $n$-ambidextrous with respect to $\Cat{C}$.
    Then one may construct a morphism $\mu^{(n)}_f \colon \id{\Cat{C}(B)} \to f_! f^\ast$ which exhibits $f_!$ as right adjoint to $f^\ast$.

    For $n = -2$ we declare that any $(-2)$-truncated morphism in $\topos Q$ is $n$-ambidextrous.
    In this case $f$, and hence the counit $f_! f^\ast \to \id{\Cat{C}(B)}$ is an equivalence and one defines $\mu^{(-2)}_f \colon \id{\Cat{C}(B)} \to f_! f^\ast$ as its inverse.

    Assume now that the collection of $n$-truncated $n$-ambidextrous morphisms is defined for some $n \ge -2$ as well as the morphisms $\mu^{(n)}_f$.

    An $(n+1)$-truncated morphism $(f\colon A \to B) \in \topos Q$ is called \defemph{weakly $(n+1)$-ambidextrous with respect to $\Cat{C}$} if the diagonal $\Delta_f \colon A \to A \times_B A$ is $n$-ambidextrous with respect to $\Cat{C}$.
    Using the diagram
    \begin{center}
        \begin{tikzcd}
            A \ar[dr, "\Delta"] \ar[drr,bend left=30, "\id{A}"] \ar[ddr,bend right=40 ,"\id{A}"] && \\
            & A \times_B A  \ar[r, "\mathrm{pr}_1"] \ar[d, "\mathrm{pr}_2"]& A \ar[d, "f"]\\
            & A \ar[r, "f"]& B
        \end{tikzcd}
    \end{center}
    we define the so-called \defemph{adjoint norm map} $\nm \colon f^\ast f_! \to \id{\Cat{C}(A)}$ as the composition
    \[\nm \colon f^\ast f_! \xrightarrow{\textup{BC}_!^{-1}} \mathrm{pr}_{1!} \mathrm{pr}_2^\ast \xrightarrow{\mu^{(n)}_\Delta} \mathrm{pr}_{1!} \Delta_! \Delta^\ast \mathrm{pr}_2^\ast \simeq \id{\Cat{C}(A)}.\]
    A $(n+1)$-truncated morphism $(f\colon A \to B) \in \topos Q$ is called \defemph{$(n+1)$-ambidextrous with respect to $\Cat{C}$} if every base change $f'$ of $f$ is weakly $(n+1)$-ambidextrous with respect to $\Cat{C}$
    and the corresponding adjoint norm map witnesses $f'_!$ as right adjoint of $f'^{\ast}$.
    Then $\mu^{(n+1)}_f \colon \id{\Cat{C}(B)} \to f_! f^\ast$ denotes a corresponding unit for the resulting adjunction $f^\ast \dashv f_!$.
\end{defn}

\begin{defn} \label{defn:recollection-ambidexterity:C-ambidexterity}
    If $(f \colon A \to B) \in \topos Q$ is (weakly) $n$-ambidextrous for some $n$ with respect to $\Cat{C}$, we will also say that $f$ is \defemph{(weakly) $\Cat{C}$-ambidextrous}.
\end{defn}

\begin{rmk}[Norm map, {\cite[Remark 4.1.12]{HopkinsLurie:Ambidexterity}}] \label{rmk:recollection-ambidexterity:norm-map}
    If $(f \colon A \to B) \in \topos Q$ is a weakly $\Cat{C}$-ambidextrous map, and $f^\ast \colon \Cat{C}(B) \to \Cat{C}(A)$ admits a right adjoint $f_\ast$, then the adjoint norm map
    $\nm_f \colon f^\ast f_! \to \id{\Cat{C}(A)}$ corresponds to a natural transformation $\Nm_f \colon f_! \to f_\ast$.
    This morphism is called the \defemph{norm map} associated to $f$.
    Then $f$ is $\Cat{C}$-ambidextrous if and only if for each base change $f'$ the restriction functor $f'^\ast$ admits a right adjoint $f'_\ast$ and the associated norm map $\Nm_{f'}$ is an equivalence.
\end{rmk}

\begin{defn}[$\topos Q$-semiadditivity, {\cite[Definition 3.10]{CnossenLenzLinskens:ParametrizedHigherSemiadditivityUniversality}}] \label{defn:recollection-ambidexterity:Q-semiadditivity}
    Let $\topos B$ be a topos with a locally inductible subcategory $\topos Q$.
    Then a $\topos B$-category $\Cat{C}$ is \defemph{$\topos Q$-semiadditive} if it admits $\topos Q$-colimits and if every truncated map in $\topos{Q}$ is $\Cat{C}$-ambidextrous.
\end{defn}

%% file: main_result.tex
\section{The class of equivariant $\pi$-finite morphisms} \label{sec:class-of-morphisms}
In this section we define an equivariant analogue of the class $(\Spc_\pi)_{\mathrm{loc}}$ of morphisms whose fibers are $\pi$-finite spaces.
Furthermore, we show that the resulting class of morphisms is a locally inductible subcategory of $\GSpc$.

\begin{defn} \label{defn:genuine-Q-setup}
    From now on we denote by $\topos{Q}$ the following wide subcategory of $\GSpc$.
    A morphism $f \colon \topos X \to \topos Y$ of $G$-spaces is in $\topos Q$ if the following hold:
    \begin{enumerate}
        \item $f$ is truncated.
        \item For every subgroup $H \le G$ and every point $y\colon \ast \to \topos Y^H$, the fiber $F_y$ of $f^H \colon \topos X^H \to \topos Y^H$ is a $\pi$-finite space.
    \end{enumerate}
\end{defn}

\begin{rmk} \label{rmk:genuine-Q-setup-alternative}
    \begin{enumerate}
        \item \cref{defn:genuine-Q-setup} can alternatively be phrased as follows: a map $f \colon \topos X \to \topos Y$ of $G$-spaces is in $\topos Q$ if for all subgroups $H \le G$ the map $f^H \colon \topos X^H \to \topos Y^H$ of spaces is in $(\Spc_\pi)_{\mathrm{loc}}$.
        \item Note that the second condition of \cref{defn:genuine-Q-setup} can equivalently be phrased as follows:
        for every subgroup $H,K \le G$, and points $p_K \colon G/K \to \topos Y$, all the $H$-fixed points of the pullback
        \begin{center}
            \begin{tikzcd}
                F_{p_K} \ar[r] \ar[d] \cartsymb & G/K \ar[d,"p_K"] \\
                \topos X \ar[r,"f"] & \topos{Y}
            \end{tikzcd}
        \end{center}
        are $\pi$-finite.
    \end{enumerate}
\end{rmk}

\begin{prop}\label{lem:genuine-Q-locally-inductible}
    The subcategory $\topos Q \subset \GSpc$ of \cref{defn:genuine-Q-setup} is locally inductible.
\end{prop}
\begin{proof}
    Let $\topos X \in \GSpc$.
    The identity $\id{}\colon \topos{ X \to X}$ is an equivalence and is therefore $(-2)$-truncated.
    Moreover, for all points $p_K \colon G/K \to \topos X$ the pullback of $\id{}$ along $p_K$ is equivalent to $G/K$, which has $\pi$-finite $H$-fixed points.
    Let $f \colon \topos X \to \topos Y$ and $g \colon \topos Y \to \topos Z$ be two composable morphisms in $\topos Q$.
    Then their composition is truncated by \cite[Example 5.2.8.16 and Proposition 5.2.8.6(3)]{Lurie:HigherTopos}.
    Checking that the fixed points of the fibers of the composition $g \circ f$ are $\pi$-finite reduces to showing that $\pi$-finite spaces are closed under extensions.
    This follows, for example, from \cite[Proposition 2.2.3]{Anel:piFiniteSpaces}.
    Hence, $\topos Q$ is a wide subcategory of $\GSpc$.

    Showing that $\topos{Q}$ is local, truncated, closed under base change and diagonals can be done after taking $H$-fixed points since all these operations commute with $(-)^H$, see \cite[Proposition 5.1.2.3]{Lurie:HigherTopos}.
    Then the claim follows from \cref{exmpl:recollection-ambidexterity:pi-finite-spaces}.
\end{proof}

\section{A fiberwise criterion for ambidexterity} \label{sec:reduction-fibers}
In their proof that $K(n)$-local spectra are $\infty$-semiadditive, Hopkins and Lurie show that ambidexterity of a map of spaces can be checked on its fibers \cite[Proposition 4.3.5]{HopkinsLurie:Ambidexterity}.
The goal of this section is to show an analogous result for maps of $G$-spaces.

\begin{lem} \label{lem:conservative-stable-projections}
    Let $\topos X$ be a $G$-space.
    For each subgroup $H \le G$, denote points of $\topos X$ by $p_H \colon G/H \to \topos X$.
    Then the family of projections \[\{p^\ast_{H} \colon \uSp(\topos X) \to \Sp_H\}_{G/H \xrightarrow{p_H} \topos X}\] is jointly conservative.
\end{lem}
\begin{proof}
    Since $\GSpc$ is a presheaf topos, \cite[Lemma 5.1.5.3]{Lurie:HigherTopos} provides an equivalence \[\topos X \simeq \colim_{(G/H \xrightarrow{p_H} \topos X) \in (\Orb)_{/\topos X}} G/H.\]
    Hence, it follows that \[\uSp(\topos X) \simeq \lim_{(G/H \xrightarrow{p_H} \topos X) \in (\Orb)_{/\topos X}} \Sp_H.\]
    This concludes the proof since projections out of a limit are jointly conservative by combining \cite[Corollaries 3.3.3.2 and 5.2.8.18]{Lurie:HigherTopos}.
\end{proof}

\begin{lem} \label{lem:conservative-stable-diagonals}
    Let $f \colon \topos X \to \topos Y$ be a map of $G$-spaces.
    Moreover, for a subgroup $H \le G$ let $p_H \colon G/H \to \topos Y$ be a point and
      \begin{center}
        \begin{tikzcd}
            F_{p_H} \ar[r,"f_{p_H}"] \ar[d, "g_{p_H}"] \cartsymb & G/H \ar[d,"p_H"] \\
            \topos X \ar[r,"f"] & \topos{Y}
        \end{tikzcd}
    \end{center}
    a pullback diagram.
    The diagonals of $f$ and $f_{p_H}$ induce a family of functors \[\{G_{p_H} \colon \uSp(\topos{X}\times_{\topos{Y}}\topos X) \to \uSp(F_{p_H} \times_{G/H} F_{p_H})\}_{ G/H \xrightarrow{p_H} \topos Y}\] which is jointly conservative.
\end{lem}
\begin{proof}
    We may view $\Delta_f \colon \topos X \to \topos{X \times_Y X}$ as a morphism in $(\GSpc)_{/\topos Y}$ if one remembers $f \colon \topos X \to \topos Y$ and $f \circ \mathrm{pr} \colon \topos{X \times_Y X} \to \topos Y$, similarly for $\Delta_{f_{p_H}} \colon F_{p_H} \to F_{p_H} \times_{G/H} F_{p_H}$.
    This yields a commutative diagram
    \begin{equation}
        \begin{tikzcd} \label{sq:main-thm:diagonals-fibers-conservativity}
            F_H \ar[rr, "g_{p_H}"] \ar[dr, "\Delta_{f_{p_H}}"] \ar[ddr, near end,"f_{p_H}"'] &  & \topos{X} \ar[dr, "\Delta_f"] \ar[ddr,near end ,"f"']&\\
            & F_H \times_{G/H} F_H \ar[rr, crossing over] \ar[d]& & \topos{X} \times_{\topos Y} \topos X \ar[d]\\
            & G/H \ar[rr, "p_H"] &  &  \topos Y \rlap{.}
        \end{tikzcd}
    \end{equation}
    It follows from \cite[Lemma 5.1.5.3]{Lurie:HigherTopos} that \[\topos Y \simeq \colim_{(G/H \xrightarrow{p_H} \topos Y) \in (\Orb)_{/\topos Y}} G/H \in \GSpc.\]
    Together with \cite[Proposition 1.2.13.8]{Lurie:HigherTopos}, this gives \[\id{\topos Y} \simeq \colim_{(G/H \xrightarrow{p_H} \topos Y) \in (\Orb)_{/\topos Y}} (G/H \xrightarrow{p_H} \topos Y) \in (\GSpc)_{/\topos Y}.\]
    Applying the pullback functor $f^\ast \colon (\GSpc)_{/ \topos Y} \to (\GSpc)_{/ \topos X}$ gives
    \[f^\ast(\id{\topos Y}) \simeq \colim_{(G/H \xrightarrow{p_H} \topos Y) \in (\Orb)_{/\topos Y}} f^\ast(G/H \xrightarrow{p_H} \topos Y) \simeq  \colim_{(G/H \xrightarrow{p_H} \topos Y) \in (\Orb)_{/\topos Y}} (F_{p_H} \xrightarrow{g_{p_H}} \topos X) \in (\GSpc)_{/\topos X}\]
    since colimits in topoi are universal by \cite[Proposition 6.1.3.10]{Lurie:HigherTopos}.
    In particular, \[\topos X \simeq \colim_{(G/H \xrightarrow{p_H} \topos Y) \in (\Orb)_{/\topos Y}} F_{p_H} \quad \textup{ and }\quad \uSp(\topos X) \simeq \lim_{(G/H \xrightarrow{p_H} \topos Y) \in (\Orb)_{/\topos Y}} \uSp(F_{p_H}).\]
    It follows that $\topos X \times_{\topos Y} \topos X \simeq \colim_{(G/H \xrightarrow{p_H} \topos Y) \in (\Orb)_{/\topos Y}} (F_{p_H} \times_{G/H} F_{p_H})$, using that colimits are universal.
    Hence, $\uSp(\topos X \times_{\topos Y} \topos X) \simeq \lim_{(G/H \xrightarrow{p_H} \topos Y) \in (\Orb)_{/\topos Y}} \uSp(F_{p_H} \times_{G/H} F_{p_H})$.
    Under this equivalence, the functors $G_{p_H}$ are precisely the projection functors. 
    They are therefore jointly conservative by the same argument as in \cref{lem:conservative-stable-projections}.
\end{proof}

\begin{thm} \label{lem:genuine-Q-fiberwise-ambidextrous}
    Let $f \colon \topos X \to \topos Y \in \topos Q$.
    Then $f$ is (weakly) $n$-ambidextrous with respect to $\RSp$ if and only if for every subgroup $H \le G$ and every point $p_H \colon G/H \to \topos Y$, the pullback $f_{p_H} \colon F_{p_H} \to G/H$ is (weakly) $n$-ambidextrous with respect to $\RSp$.
\end{thm}
\begin{proof}
    The ''only if''-direction is immediate from the definition.

    We prove the ''if''-direction by induction on $n$.
    Let $n =-2$.
    Then $f$ is an equivalence and there is nothing to prove.
    Let $n \ge -2$ and assume that the statement holds for $n$-truncated maps.
    We first show the claim for weakly ambidextrous morphisms.
    Assume that, for every subgroup $H \le G$ and every point $p_H \colon G/H \to \topos Y$, the morphisms $f_{p_H} \colon F_{p_H} \to G/H$ are weakly $(n+1)$-ambidextrous.
    We need to show that the diagonal $\Delta_f \colon \topos X \to \topos{X \times_Y X}$ is $n$-ambidextrous.
    Applying $\RSp(-)$ to the top square of \cref{sq:main-thm:diagonals-fibers-conservativity} yields a commutative diagram
    \begin{center}
        \begin{tikzcd}
            \RSp(F_{p_H}) & \RSp(\topos X) \ar[l] \\
            \RSp(F_{p_H} \times_{G/H} F_{p_H}) \ar[u, "\Delta_{f_{p_H}}^\ast"] & \RSp(\topos X \times_{\topos Y} \topos X) \ar[l] \ar[u, "\Delta_f^\ast"] \rlap{.}
        \end{tikzcd}
    \end{center}
    Both $\Delta_{f_{p_H}}^\ast$ and $\Delta_f^\ast$ have a left and a right adjoint (since $\RSp$ is $G$-presentable), and the usual Beck-Chevalley transformations, see e.g., \cite[§4.1]{HopkinsLurie:Ambidexterity}, yield commutative diagrams
    \begin{center}
        \begin{tikzcd}
            \RSp(F_{p_H}) \ar[d, "\Delta_{{f_{p_H}},!}"] & \RSp(\topos X) \ar[l] \ar[d, "\Delta_{f,!}"] & \RSp(F_{p_H}) \ar[d, "\Delta_{{f_{p_H}},\ast}"] & \RSp(\topos X) \ar[l] \ar[d, "\Delta_{f,\ast}"]\\
            \RSp(F_{p_H} \times_{G/H} F_{p_H})  & \RSp(\topos X \times_{\topos Y} \topos X) \ar[l] & \RSp(F_{p_H} \times_{G/H} F_{p_H})  & \RSp(\topos X \times_{\topos Y} \topos X) \ar[l] \rlap{.}
        \end{tikzcd}
    \end{center}
    By the induction hypothesis, $\Delta_f$ is weakly $n$-ambidextrous. This gives the norm map ${\Nm_{\Delta_f} \colon \Delta_{f,!} \to \Delta_{f,\ast}}$.
    But this is an equivalence on objects by the commutativity of the above diagrams, the assumption that all $\Delta_{f_{p_H}}$ are $n$-ambidextrous, and \cref{lem:conservative-stable-diagonals}.
    
    Assume that for all subgroups $H \le G$ and points $p_H \colon G/H \to \topos Y$, the morphisms $f_{p_H} \colon F_{p_H} \to G/H$ are $(n+1)$-ambidextrous with respect to $\RSp$.
    Consider the diagram
    \begin{center}
        \begin{tikzcd}
            \RSp(F_{p_H}) & \Sp_{H,\Q} \ar[l, "f_{p_H}^\ast"']\\
            \RSp(\topos{X}) \ar[u]& \RSp(\topos{Y}) \ar[l, "f^\ast"] \ar[u, "p_H^\ast"'] \rlap{,}
        \end{tikzcd}
    \end{center}
    obtained by applying $\RSp$ to the pullback square 
    \begin{center}
        \begin{tikzcd}
            F_{p_H} \ar[r,"f_{p_H}"] \ar[d] \cartsymb & G/H \ar[d, "p_H"] \\
            \topos X \ar[r,"f"] & \topos{Y} \rlap{.}
        \end{tikzcd}
    \end{center}
    The same reasoning as for the first part (using \cref{lem:conservative-stable-projections} instead of \cref{lem:conservative-stable-diagonals}) yields the claim.
\end{proof}

\begin{cor} \label{lem:genuine-Q-fiberwise-norms}
    Let $(f \colon \topos X \to \topos Y) \in \topos Q$ be a weakly $\RSp$-ambidextrous map of $G$-spaces.
    For every subgroup $H \le G$ let $p_H \colon G/H \to \topos Y$ be a point of $\topos Y$. 
    Denote the pullback of $f$ along $p_H$ by $f_{p_H} \colon F_{p_H} \to G/H$.
    If the norm map \[\Nm_{f_{p_H}} \colon f_{p_H,!} \to f_{p_H,\ast}\] is an equivalence for each $f_{p_H}$, then the norm of $f$\[\Nm_f \colon f_! \to f_\ast\] is an equivalence.
\end{cor}
\begin{proof}
    This is immediate from the second part of the proof of \cref{lem:genuine-Q-fiberwise-ambidextrous}.
\end{proof}

\begin{prop} \label{lem:ambidexterity:reduction-to-G/G}
    Let $H \le G$ be a subgroup and $(f \colon \topos X \to G/H) \in \topos Q_G$ a weakly $\RSp$-ambidextrous map of $G$-spaces.
    Moreover, let $\tilde{f} \colon \topos X_H \to H/H$ be the morphism of $H$-spaces corresponding to $f$ under the equivalence $(\GSpc)_{/{(G/H)}} \simeq \Spc_H$.
    Assume that $\tilde{f}$ is weakly $\myuline{\Sp}_{H,\Q}$-ambidextrous.
    Then the norm map \[\Nm_f \colon f_! \to f_\ast \colon \RSp(\topos X) \to \RSp(G/H) \simeq \Sp_{H,\Q}\] is an equivalence if and only if the norm map \[\Nm_{\tilde{f}} \colon \tilde{f}_! \to \tilde{f}_\ast \colon \myuline{\Sp}_{H,\Q}(\topos X_H) \to \myuline{\Sp}_{H,\Q}(H/H) \simeq \Sp_{H,\Q}\] is an equivalence.
\end{prop}
\begin{proof}
    The equivalence $(\GSpc)_{/{(G/H)}} \simeq \Spc_H \simeq (\Spc_H)_{/\ast}$ is induced by taking the fiber over $eH$.
    Under this equivalence $f$ gets sent to $\tilde{f} \colon \topos X_H \to H/H$, where $\topos X_H$ denotes the pullback of $f$ along the point $eH$.
    Since this equivalence of topoi induces an equivalence $\textup{Cat}(\Spc_H) \simeq \textup{Cat}((\Spc_G)_{/(G/H)})$, \cite[Lemma 2.1.18]{CnossenLenzLinskens:ParametrizedStabilityUPGlobalSpectra} implies that $\Nm_f$ is an equivalence if and only if $\Nm_{\tilde{f}}$ is an equivalence.
\end{proof}

\section{Equivariant semiadditivity of rational $G$-spectra} \label{sec:main-thm}
We now prove that the $G$-category of rational genuine $G$-spectra is $\topos Q$-semiadditive.
First, we show that the norm of any map $\topos X \to \ast$ in $\topos Q$ is an equivalence.
This case combined with the results of \cref{sec:reduction-fibers} will be the main tools for the proof of \cref{thm:main-theorem}.

\begin{lem} \label{lem:main-thm:norm-to-point}
    Let $(f \colon \topos X \to \ast) \in \topos Q$ be a weakly $\RSp$-ambidextrous map.
    Then the norm map \[\Nm_f \colon f_! \to f_\ast \colon \RSp(\topos X) \to \Sp_{G,\Q}\] is an equivalence.
\end{lem}
\begin{proof}
    Since the family $\{\Phi^H\}_{H \le G}$ of geometric fixed points is jointly conservative, it suffices to show that for each $H \le G$ the map \[\Phi^H \Nm_f \colon \Phi^H f_! \to \Phi^H f_\ast \colon \RSp(\topos X) \to \Sp_\Q\] is an equivalence.
    Note that $\Phi^H$ preserves rationality by \cref{lem:geometric-fixed-points:rational-to-rational}.
    For $H\leq G$, restriction to $H$ preserves $G$-limits and $G$-colimits and, by \cite[Proposition 3.7(2)]{CnossenLenzLinskens:ParametrizedHigherSemiadditivityUniversality}, carries the norm map $\Nm_f$ to the norm map of the underlying map $\Res{G}{H}\topos X \to \ast$ of $H$-spaces.
    Since $\Phi^H = \Phi^H \circ \Res{G}{H}$, it suffices to do the argument below for the top group $G$.
    From \cite[Proposition 3.7(2)]{CnossenLenzLinskens:ParametrizedHigherSemiadditivityUniversality} we obtain a commutative diagram
    \begin{equation} \label{sq:main-thm:norm-classical}
        \begin{tikzcd}
            f^G_! \Phi^G \ar[r, "\Nm_{f^G}"] \ar[d, "\mathrm{BC}_!"] & f^G_\ast \Phi^G \\
            \Phi^G f_! \ar[r, "\Phi^G \Nm_f"] & \Phi^G f_\ast \ar[u, "\mathrm{BC}_\ast"'] \rlap{.}
        \end{tikzcd}
    \end{equation}
    The left vertical morphism is an equivalence by construction and using the description of \cite[Observation 4.2.2]{HilmanKirsteinKremer:ParametrisedPoincareDualityEquivariant}.
    The right vertical morphism is an equivalence by \cref{lem:geometric-fixed-points:parametrized-geometric-rationally-limits}.
    The top arrow is thus the classical norm map, which is an equivalence by \cite[Theorem 5.2.1]{HopkinsLurie:Ambidexterity}, since $\topos X^G$ is a $\pi$-finite space by assumption.
    In particular, $\Phi^G \Nm_f$ is an equivalence.
\end{proof}

\begin{thm} \label{thm:main-theorem}
    The $G$-category $\RSp \colon \op{\GSpc} \to \Cati$ is $\topos Q$-semiadditive.
\end{thm}
\begin{proof}
    We prove the claim for all finite groups simultaneously.
    The subcategory $\topos Q \subset \GSpc$ is locally inductible by \cref{lem:genuine-Q-locally-inductible}.
    In order to prove that $\RSp$ is $\topos Q$-semiadditive, one needs to show that $\RSp$ is $\topos{Q}$-cocomplete and every morphism in $\topos{Q}$ is $\RSp$-ambidextrous.
    The $G$-category $\RSp$ is $G$-presentable by \cref{exmpl:parametrized:rational-G-spectra}.
    So it is $\topos Q$-cocomplete as well.
    
    The proof of ambidexterity is by induction on $n$.
    For $n = -2$, every $(-2)$-truncated morphism in $\topos{Q}$ was defined to be $(-2)$-ambidextrous.
    Assume that every $n$-truncated morphism in $\topos{Q}$ is $n$-ambidextrous.

    Let $(f \colon \topos X \to \topos Y) \in \topos Q$ be $(n+1)$-truncated.
    Then the diagonal $\Delta_f \colon \topos{X} \to \topos{X} \times_\topos{Y} \topos{X}$ is $n$-truncated and by \cref{lem:genuine-Q-locally-inductible} in $\topos Q$.
    Therefore, it is $n$-ambidextrous by the induction hypothesis.
    It follows that $f$ is weakly $(n+1)$-ambidextrous, and we denote its adjoint norm map by \[\nm_f \colon f^\ast f_! \to \id{\RSp(\topos{X})}.\]
    To show that $f$ is $(n+1)$-ambidextrous, we need to prove that for any base change $f'$ of $f$ the adjoint norm map $\nm_{f'}$, which exists by the above, exhibits $f'_!$ as a right adjoint of $f'^\ast$.
    Since $\RSp$ is $G$-presentable, every restriction functor $f'^{\ast}$ admits a right adjoint $f'_\ast$, and we may instead show the norm map $\Nm_{f'} \colon f'_! \to f'_\ast$ is an equivalence, see \cite[Remark 4.1.12]{HopkinsLurie:Ambidexterity}.
    Let $f' \colon \topos X' \to \topos Y'$ be some base change of $f$.
 
    By \cref{lem:genuine-Q-fiberwise-norms} it suffices to show that for every subgroup $H \le G$ and every point $p_H \colon G/H \to \topos Y'$ the norm map of the pullback $f'_{p_H} \colon F_{p_H} \to G/H$ is an equivalence.
    The case $H = G$ follows immediately from \cref{lem:main-thm:norm-to-point}.

    Let $H < G$ be a proper subgroup. 
    For each of the maps $f'_{p_H} \colon F_{p_H} \to G/H$ above, we must show that the norm $\Nm_{f'_{p_H}} \colon f'_{p_H,!} \to f'_{p_H,\ast}$ is an equivalence.
    Under the equivalence $(\GSpc)_{/{(G/H)}} \simeq \Spc_H$ the $(n+1)$-truncated map $f'_{p_H} \colon F_{p_H} \to G/H$ of $G$-spaces is sent to the $(n+1)$-truncated map $\tilde{f'}_{p_H} \colon (F_{p_H})_H \to H/H$ of $H$-spaces.
    Moreover, $\tilde{f'}_{p_H}$ belongs to $\topos Q_H \subset \Spc_H$ since for $K \le H$ the fibers of $((F_{p_H})_H)^K \to (H/H)^K$ are $\pi$-finite because the fibers of $(f'_{p_H} )^K$ are so by assumption.
    In particular, the diagonal $\Delta_{\tilde{f'}_{p_H}}$ is an $n$-truncated morphism of $H$-spaces.
    By induction (for the group $H$), it follows that $\Delta_{\tilde{f'}_{p_H}}$ is $n$-ambidextrous with respect to the $H$-category $\myuline{\Sp}_{H,\Q}$.
    Thus, $\tilde{f'}_{p_H}$ is weakly $(n+1)$-ambidextrous with respect to $\myuline{\Sp}_{H,\Q}$.
    \cref{lem:main-thm:norm-to-point} shows that the norm $\Nm_{\tilde{f'}_{p_H}}$ is an equivalence.
    Finally, \cref{lem:ambidexterity:reduction-to-G/G} then implies that the norm $\Nm_{f'_{p_H}}$ is an equivalence.

    Hence, the norm of each base change of $f$ is an equivalence. 
    It follows that $f$ is $(n+1)$-ambidextrous with respect to $\RSp$.
\end{proof}